%% file: main.tex
\documentclass[a4paper,11pt]{article}
\usepackage{preamble}
\title{A Regression-Based Prediction-Correction Method for Stochastic Time-Varying Optimization Problems}
\author[1]{Tomoya Kamijima}
\author[1]{Naoki Marumo}
\author[1,2]{Akiko Takeda}
\affil[1]{Department of Mathematical Informatics, The University of Tokyo, Tokyo, Japan.}
\affil[2]{RIKEN Center for Advanced Intelligence Project, Tokyo, Japan.}
\date{}

\begin{document}
\maketitle

\begin{abstract}
    In many real-world applications, optimization problems evolve continuously over time and are often subject to stochastic noise. 
    We consider a stochastic time-varying optimization problem in which the objective function $f(\bm x;t)$ changes continuously and only noisy gradient observations are available. 
    In deterministic settings, the prediction-correction method that exploits the time derivative of the solution is effective for accurately tracking the solution trajectory. 
    However, a straightforward extension to stochastic problems requires an estimate of $\nabla_{\bm xt} f(\bm x;t)$ and the computation of a Hessian inverse at each step--requirements that are difficult or costly in practice. 
    To address these issues, we propose a prediction-correction algorithm that uses a regression-based prediction step: the prediction is formed as a linear combination of recent iterates, which can be computed efficiently without estimating $\nabla_{\bm xt}f(\bm x;t)$ or computing Hessian inversions. 
    We prove a tracking-error bound for the proposed method under standard smoothness and stochastic assumptions. 
    Numerical experiments show that the regression-based prediction improves tracking accuracy while reducing computational cost compared with existing methods.
\end{abstract}

\input{1_introduction.tex}

\input{2_problem.tex}
\input{3_proposed.tex}
\input{4_tracking.tex}

\input{5_alpha.tex}

\input{6_numerical.tex}
\input{7_conclusion.tex}

\section*{Acknowledgments}
This work was partially supported by JSPS KAKENHI (23H03351, 24K23853, and 25KJ1123) and JST CREST (JPMJCR24Q2).

\bibliographystyle{abbrvnat}
\bibliography{ref}

\appendix

\input{9_appendix.tex}

\end{document}

%% file: 1_introduction.tex
\section{Introduction}
In many real-world applications, optimization problems are not static but continuously change over time.
Such problems arise in domains like robotics \citep{derenick2009optimal,zavlanos2012network}, control \citep{cao2012online,paternain2019prediction}, online principal component analysis \citep{feng2013online,bouwmans2014robust,lois2015online}, data analysis, and optimal power flow \citep{dall2016optimal}.
More applications are given in \citet{dall2020optimization} and \citet{simonetto2020time}.
Thus, it is important to consider the time-varying optimization problem
\begin{align}
    \min_{\bm x\in\R^d}f(\bm x;t),\label{eq:problem_deterministic}
\end{align}
where $t\geq0$ is the time variable.
Problem~\cref{eq:problem_deterministic} has been actively studied since \citet{simonetto16} proposed a prediction-correction framework.
After that, various prediction-correction algorithms have been developed \citep{simonetto17,bastianello19,lin2019simplified,bastianello23,kamijima2025simple}.
In these studies, it is assumed that the function value $f(\bm x;t)$ and its gradient $\nabla_{\bm x} f(\bm x;t)$ can be computed exactly.

This paper focuses on the following stochastic variant:
\begin{gather}
    \min_{\bm x\in\R^d}f(\bm x;t)\coloneqq\E_{\bm y\sim \mathcal P(t)}[\ell(\bm x;\bm y)]+\psi(\bm x)\label{eq:problem_intro}
\end{gather}
where $\mathcal P(t)$ is an unknown probability distribution on $\R^{d_y}$ that depends on the time $t$, a function $\ell\colon\R^d\times\R^{d_y}\to\R$ is a loss function, and $\psi\colon\R^d\to\R\cup\{\infty\}$ is a regularization function.
Prediction-correction algorithms have also been proposed for this stochastic problem \citep{maity2022predictor,simonetto2024nonlinear}.
However, the stochastic prediction-correction-based method proposed by \citet{maity2022predictor} requires an estimation of $(\nabla_{\bm x\bm x}f(\bm x;t))^{-1}\nabla_{\bm xt}f(\bm x;t)$, which is often difficult to obtain in practice.
The method proposed by \citet{simonetto2024nonlinear} lacks theoretical guarantees that incorporating the prediction mechanism yields performance improvements, although empirical results show improved performance.

In order to resolve the issues in these previous studies, we generalize the SHARP algorithm~\citep{kamijima2025simple} to the stochastic setting.
SHARP is a prediction-correction algorithm that computes the prediction by using the Lagrange interpolation and thus does not use derivative information in the prediction.
However, randomness makes the Lagrange interpolation unstable because the previous solutions contain noise.
To address this problem, we utilize the regression instead of the Lagrange interpolation to compute the prediction.
The regression finds a model that minimizes the squared error between the previous solutions and the points computed by the model, while the Lagrange interpolation fits the previous solutions exactly.
Thus, the regression is expected to be more robust to noise than the Lagrange interpolation.

Our contributions are summarized as follows:
\begin{itemize}
    \item We propose a stochastic SHARP algorithm, which does not need estimations of time derivatives and Hessian matrices.
    \item The algorithm tracks a target trajectory with high accuracy under suitable assumptions. The method achieves the $O(h^{q/(2q+1)})$ tracking error bound with $O(1)$ gradient calls per round, where $h$ is a sampling period and $q$ is an integer such that the solution trajectory is $q$-times differentiable and the derivative is bounded.
    See \cref{tab:error_bound} for a comparison with existing methods.
    \item We conduct numerical experiments to validate the effectiveness of the proposed algorithm.
\end{itemize}


\subsection{Related Work}

\subsubsection*{Non-stochastic Time-Varying Optimization}
\citet{popkov05} solved Problem~\cref{eq:problem_deterministic} by considering the problem as a sequence of stationary problems by discretizing the time variable $t$ and iteratively applying the gradient descent.
The tracking error of the method is theoretically guaranteed to be bounded.
\citet{simonetto16} proposed a prediction-correction framework, which allows us to give a smaller tracking error bound.
With this framework as the foundation, several generalizations have been introduced.
\citet{simonetto17} tackled a constrained problem by using a projection operator, and \citet{bastianello19} generalized it by applying the splitting methods.
Besides, extrapolation-based prediction-correction algorithms have been proposed \citep{lin2019simplified,bastianello23,kamijima2025simple}.

\subsubsection*{Stochastic Time-Varying Optimization}
\citet{dixit2019online} proposed to apply the proximal gradient descent to Problem~\cref{eq:problem_intro} and proved that the tracking error is bounded.
\citet{cutler2021stochastic} dealt with the same algorithm and gave a better tracking error bound.
Following these studies, prediction-correction algorithms have been proposed to improve the bound just as in the non-stochastic setting.
\citet{maity2022predictor} added a prediction step similar to that of \citet{simonetto16} in the setting where $\psi=0$.
However, the method tracks the solution accurately only when an estimation of $(\nabla_{\bm x\bm x}f(\bm x;t))^{-1}\nabla_{\bm xt}f(\bm x;t)$ with small error is available, which is restrictive because it is difficult to estimate time derivatives in general stochastic settings.
\citet{simonetto2024nonlinear} proposed a different prediction strategy inspired by the extended Kalman filter.
Although the method performs well in the numerical experiments, the outperformance is not theoretically guaranteed.





\subsubsection*{Online Convex Optimization and Dynamic Regret Analysis}
Online convex optimization also deals with time-varying optimization problems.
The main difference from the time-varying optimization problem we are considering is that online optimization does not assume that the objective function changes continuously over time.
\citet{zinkevich2003online} proposed an online gradient descent algorithm and derived regret bounds.
Since then, numerous studies have been conducted to improve the regret bounds and to develop new algorithms.
A problem closely related to our work is regularized loss minimization in the online setting \citep{xiao2009dual,duchi2009efficient,duchi2010composite,suzuki2013dual}.
Another relevant line of research is dynamic regret analysis, which focuses on the time-varying comparator \citep{zinkevich2003online,hall2013dynamical,besbes2015non,zhang2018adaptive,nonhoff2020online,zhao2024adaptivity}.
The dynamic regret is $\Omega(T)$ in the worst case \citep[Proposition 1]{besbes2015non}, but it is possible to achieve sublinear dynamic regret bounds under some conditions, for example, the path length growth is sufficiently slow \citep{zinkevich2003online,zhang2018adaptive,zhao2024adaptivity}.

\subsection{Notation}
For a $q$-times differentiable function $\bm \varphi\colon\R\to\R^d$, we denote its $i$-th derivative by $\bm \varphi^{(i)}$ for $i=1,\ldots,q$.
The norm $\|\cdot\|_p$ denotes the $\ell_p$-norm and we use $\|\cdot\|$ to denote the Euclidean norm $\|\cdot\|_2$.
The all-one vector is denoted by $\bm{1}\coloneqq(1,\ldots,1)^\top$.
The norm $\|\cdot\|_{\mathrm F}$ denotes the Frobenius norm.
For a convex function $\psi\colon\R^d\to\R$, the subdifferential at $\bm x$ is defined by $\partial\psi(\bm x)\coloneqq\Set{\bm g\in\R^d}{\forall \bm y\in\R^d,~\psi(\bm y)\geq\psi(\bm x)+\inner{\bm g}{\bm y-\bm x}}$.
For a closed convex set $\mathcal X\subset\R^d$, the projection operator is defined by $\proj_{\mathcal X}(\bm x)\coloneqq\argmin_{\bm y\in\mathcal X}\|\bm y-\bm x\|$.
For a proper, closed, and convex function $\psi\colon\R^d\to\R$, the proximal operator is defined by $\prox_\psi(\bm x)\coloneqq\argmin_{\bm y\in\R^d}\{\psi(\bm y)+\frac12\|\bm y-\bm x\|^2\}$.
For two distributions $\mathcal P$ and $\mathcal Q$ on $\R^d$, the Wasserstein-1 distance is defined by
\begin{gather}
    W_1(\mathcal P,\mathcal Q)\coloneqq\inf_{\pi\in\Pi(\mathcal P,\mathcal Q)}\int_{\R^d\times\R^d}\|\bm x-\bm y\|d\pi(\bm x,\bm y),
\end{gather}
where $\Pi(\mathcal P,\mathcal Q)$ is the set of all couplings between $\mathcal P$ and $\mathcal Q$.

\begin{table}
    \centering
    \caption{Comparison of stochastic time-varying optimization algorithms. The columns Gradient Call and Hessian Call represent the number of gradient and Hessian evaluations per round, respectively. The column $\psi$ indicates whether the method can handle a non-smooth regularization term $\psi$. The column Error shows the bound on the tracking error defined as $\limsup_{k\to\infty}\E[\|\bm x_k-\bm x^\ast(t_k)\|]$. The parameter $h$ is a sampling period and $q$ is an integer such that the solution trajectory is $q$-times differentiable and the derivative is bounded. Note that $h^r$ decreases when $r$ increases since $h$ is small.}
    \begin{tabular}{@{\hskip 2pt}c@{\hskip 2pt}|@{\hskip 5pt}c@{\hskip 5pt}c@{\hskip 5pt}c@{\hskip 5pt}c@{\hskip 2pt}}
        \toprule
        Algorithm & Gradient Call & Hessian Call & $\psi$ & Error \\
        \midrule
        TVSGD~\citep{cutler2021stochastic} & $1$ & - & \checkmark & $O(h^{1/3})$ \\
        Stochastic PC~\citep{maity2022predictor} & $O(h^{-4/5})$ & $1$ & - & $O(h^{6/5})$\footnotemark \\
        TV-EKF~\citep{simonetto2024nonlinear} & $O(h^{-4/5})$ & $O(1)$ & - & - \\
        TV-Contract~\citep{simonetto2024nonlinear} & $O(h^{-4/5})$ & - & \checkmark & $O(1)$ \\
        \textbf{Stochastic SHARP (Ours)} & $O(1)$ & - & \checkmark & $O(h^{q/(2q+1)})$ \\
        \bottomrule
    \end{tabular}
    \label{tab:error_bound}
\end{table}
\footnotetext{This bound is valid only for least square problems. See \citet[Section 4.1]{maity2022predictor} for the details.}

%% file: 2_problem.tex
\section{Problem Setting and Assumptions}
We consider the following stochastic time-varying optimization problem:
\begin{gather}
    \min_{\bm x\in\R^d}f(\bm x;t)\coloneqq\E_{\bm y\sim \mathcal P(t)}[\ell(\bm x;\bm y)]+\psi(\bm x),\label{eq:problem}
\end{gather}
where $\mathcal P(t)$ is an unknown probability distribution on $\R^{d_y}$ that depends on a continuous-time variable $t\geq0$, a function $\ell\colon\R^d\times\R^{d_y}\to\R$ is a loss function which is strongly convex and smooth in $\bm x$, and $\psi\colon\R^d\to\R\cup\{\infty\}$ is a proper, closed, and convex regularization function. 
Our goal is to track a target trajectory $\bm x^\ast(t)\coloneqq\argmin_{\bm x\in\R^d}f(\bm x;t)$ with high accuracy.
Note that, by the strong convexity of $f(\cdot;t)$, the optimal solution $\bm x^\ast(t)$ is unique.
To achieve the goal, we sample $\bm y_k\sim \mathcal P(t_k)$ at each time $t_k\coloneqq kh$ for $k=0,1,\ldots$, where $h>0$ is a fixed sampling period, and iteratively solve the following optimization problem:
\begin{gather}
    \min_{\bm x\in\R^d}\hat f_k(\bm x)\coloneqq\hat \ell_k(\bm x)+\psi(\bm x),\quad\text{where}\quad\hat\ell_k(\bm x)\coloneqq\ell(\bm x;\bm y_k),\label{eq:problem_hat}
\end{gather}
which is an empirical approximation of \eqref{eq:problem} at time $t_k$.
If multiple samples are available, we can use the average of the loss function as $\hat\ell_k(\bm x)$ instead.
The solution of \eqref{eq:problem_hat} is denoted by $\hat {\bm x}_k^\ast\coloneqq\argmin_{\bm x\in\R^d}\hat f_k(\bm x)$, which is unique due to the strong convexity of $\hat f_k$.

We focus on the case where the problem \eqref{eq:problem} satisfies the following standard assumptions:

\begin{assumption}\label{ass:problem}
    Let $\sigma_1\geq0$, $\varepsilon>0$, $\mu>0$, and $L_{2,0}>0$ be constants.
    \begin{enuminasm}
        \item $\E_{\bm y\sim \mathcal P(t)}[\ell(\bm x;\bm y)]$ and $\E_{\bm y\sim \mathcal P(t)}[\|\nabla_{\bm x}\ell(\bm x;\bm y)\|]$ exist for all $\bm x\in\R^d$ and $t\geq0$.\label{ass:existence}
        \item $\E_{\bm y\sim \mathcal P(t_k)}\bracket*{\mynorm*{\nabla \ell(\bm x_k^\ast;\bm y)-\E_{\bm z\sim \mathcal P(t_k)}[\nabla \ell(\bm x_k^\ast;\bm z)]}^2}\leq\sigma_1^2$ for all $k\geq0$.\label{ass:gradient_variance}
        \item\label{ass:strong_convexity} $\ell(\cdot,\bm y)$ is $\mu$-strongly convex for any $\bm y\in\R^{d_y}$, i.e., for all $\bm x,\bm x'\in\R^d$ and $\bm y\in\R^{d_y}$,
        \begin{align}
            \ell(\bm x;\bm y)&\geq\ell(\bm x';\bm y)+\langle\nabla_{\bm x}\ell(\bm x';\bm y),\bm x-\bm x'\rangle+\frac{\mu}{2}\|\bm x-\bm x'\|^2.\label{eq:strong_convexity}
        \end{align}
        \item\label{ass:smoothness} $\ell(\cdot,\bm y)$ is $L_{2,0}$-smooth in $\bm x$ for any $\bm y\in\R^{d_y}$, i.e., for all $\bm x,\bm x'\in\R^d$ and $\bm y\in\R^{d_y}$,
        \begin{align}
            \ell(\bm x;\bm y)&\leq\ell(\bm x';\bm y)+\langle\nabla_{\bm x}\ell(\bm x';\bm y),\bm x-\bm x'\rangle+\frac{L_{2,0}}{2}\|\bm x-\bm x'\|^2.\label{eq:smoothness}
        \end{align}
        \item $\psi$ is a proper, closed, and convex.\label{ass:ccp}
    \end{enuminasm}
\end{assumption}

\begin{example}
    [Time-varying LASSO]
    Let $\bm y\in\R^{d_y}$ follow the time-varying distribution $\mathcal P(t)$, and it can be observed through sampling.
    Given a fixed matrix $A\in\R^{d_y\times d}$ with $d_y\geq d$ which is of full column rank, we aim to find $\bm x$ such that $\bm y$ and $A\bm x$ become close while preserving the sparsity of $\bm x$.
    The parameter $\bm x$ is computed via LASSO:
    \begin{align}
        \min_{\bm x\in\R^d}\E_{\bm y\sim\mathcal P(t)}\bracket*{\frac12\|A\bm x-\bm y\|^2}+\lambda\|\bm x\|_1,
    \end{align}
    where $\lambda\geq0$ is a regularization parameter.
    This problem corresponds to the case with $\ell(\bm x;\bm y)=\frac12\|A\bm x-\bm y\|^2$ and $\psi(\bm x)=\lambda\|\bm x\|_1$.
    The function $\ell$ is $\sigma_{\mathrm{min}}(A)^2$-strongly convex and $\sigma_{\mathrm{max}}(A)^2$-smooth, where $\sigma_{\mathrm{min}}(A)$ and $\sigma_{\mathrm{max}}(A)$ are the minimum and maximum singular values of $A$, respectively.
\end{example}



%% file: 3_proposed.tex
\section{Proposed Method}\label{sec:proposed}
This section presents Stochastic SHARP, which is an extension of SHARP \citep{kamijima2025simple} to the stochastic time-varying optimization problem~\cref{eq:problem}.
The algorithm is given in \cref{alg:proposed}, which consists of two steps: a prediction step (\cref{line:prediction}) and a correction step (\cref{line:correction}).
In the prediction step, we predict the target point $\bm x^\ast(t_k)$ by $\hat{\bm x}_k$ based on the previous solutions $\bm x_{k-1},\bm x_{k-2},\ldots$.
In the correction step, we find an approximate solution $\bm x_k$ of the optimization problem \eqref{eq:problem_hat} using the predicted point $\hat{\bm x}_k$ as an initial point.


The prediction step is based on a regression (see, e.g., \citet[Section 2]{shawe2004kernel}).
Let $\bm \phi\colon\R\to\R^p$ be a vector-valued function whose components are basis functions, and we approximate the target point $\bm x^\ast(t_k)$ by a linear combination of the basis functions:
\begin{gather}
    \bm x^\ast(t_k)\approx\hat{\bm x}_k\coloneqq\Theta_k\bm\phi(0),\label{eq:regression}
\end{gather}
where $\Theta_k\in\R^{d\times p}$ is a parameter matrix.
For example, we can choose $\bm\phi(t)=(1,t,t^2,\ldots,t^{p-1})^\top$ as a polynomial basis.
This choice approximates the optimal solution $\bm x^\ast(t)$ by a polynomial function of time.

We assume that $\Phi\coloneqq\mqty(\bm\phi(-1)&\cdots&\bm\phi(-n))\in\R^{p\times n}$ has full row rank and estimate $\Theta_k$ by the following least-squares problem:
\begin{gather}
    \Theta_k\coloneqq\argmin_{\Theta\in\R^{d\times p}}\|X_k-\Theta\Phi\|_{\mathrm F}
    =X_k\Phi^\top(\Phi\Phi^\top)^{-1},\label{eq:least_squares}
\end{gather}
where $n\geq p$ is the number of points used in the regression and $X_k\coloneqq\mqty(\bm x_{k-1} & \cdots & \bm x_{k-n})\in\R^{d\times n}$.
Plugging \eqref{eq:least_squares} into \eqref{eq:regression} gives the prediction
\begin{gather}
    \hat{\bm x}_k=X_k\bm\alpha,\quad\text{where}\quad\bm\alpha\coloneqq\Phi^\top(\Phi\Phi^\top)^{-1}\bm\phi(0)\in\R^n.\label{eq:predicted_point}
\end{gather}
For the calculation of $\bm\alpha$ using the polynomial basis, see \cref{sec:tracking_error_polynomial}.

The following directly follows from \cref{eq:predicted_point}:
\begin{align}
    \Phi\bm\alpha=\bm\phi(0).\label{eq:sum_of_alpha}
\end{align}
Especially, $\inner{\bm1}{\bm\alpha}=1$ holds if $\phi_1(t)=1$, where $\phi_1(t)$ is the first element of $\bm\phi(t)$. 
This is desirable because it ensures that the predicted point $\hat{\bm x}_k$ is consistent with scaling and translation, i.e., substituting $\bm x_{k-i}$ with $c\bm x_{k-i}+\bm v$ for all $i\in\{1,\ldots,n\}$, where $c\in\R$ and $\bm v\in\R^d$, results in $\hat{\bm x}_k$ being changed to $c\hat{\bm x}_k+\bm v$.

The correction step approximately solves the optimization problem \eqref{eq:problem_hat} to correct the predicted point $\hat{\bm x}_k$ to $\bm x_k$.
For example, we can use the proximal gradient method:
\begin{gather}
    \bm x_k^0=\hat{\bm x}_k,\quad \bm x_k^{c+1}=\prox_{\beta \psi}(\bm x_k^c-\beta\nabla\hat \ell_k(\bm x_k^c)),\quad c=0,1,\ldots,c_{\max}-1,\label{eq:proximal_gradient_method}
\end{gather}
where $\beta>0$ is a step size and $c_{\max}\in\N$ is a predetermined iteration number.

\begin{algorithm}[tb]
    \caption{Stochastic SHARP}
    \label{alg:proposed}
    \begin{algorithmic}[1]
        \Require $\bm x_0\in\R^d$, $n\in\N$, $\bm \phi\colon\R\to\R^p$
        \State $\bm x_{-n+1}=\cdots=\bm x_{-1}=\bm x_0$
        \State $\bm\alpha=\Phi^\top(\Phi\Phi^\top)^{-1}\bm\phi(0)$, where $\Phi\coloneqq\mqty(\bm\phi(-1)&\cdots&\bm\phi(-n))$
        \For{$k=1,2,\ldots$}
            \State $\hat{\bm x}_k=X_k\bm\alpha$
            \Comment {Prediction}\label{line:prediction}
            \State Set $\bm x_k$ to be an approximate solution of \cref{eq:problem_hat} by e.g., \cref{eq:proximal_gradient_method}
            \Comment {Correction}\label{line:correction}
        \EndFor
    \end{algorithmic}
\end{algorithm} 

%% file: 4_tracking.tex
\section{Tracking Error Analysis}\label{sec:tracking_error}
This section analyzes the tracking error of the proposed method.
Recall $\bm x^\ast(t)\coloneqq\argmin_{\bm x\in\R^d}f(\bm x;t)$ is not a random variable, $\hat{\bm x}_k^\ast\coloneqq \argmin_{\bm x\in\R^d}\hat f_k(\bm x)$ is a random variable which depends only on a random sample $\bm y_k\sim\mathcal P(t_k)$, $\hat{\bm x}_k$ is the predicted point, which is a random variable, and $\bm x_k$ is the corrected point, which is also a random variable.
Let $\bm x_k^\ast\coloneqq \bm x^\ast(t_k)$ and $X_k^\ast\coloneqq\mqty(\bm x_{k-1}^\ast&\cdots&\bm x_{k-n}^\ast)\in\R^{d\times n}$ for simplicity.
Our goal is to give an upper bound on the asymptotic tracking error defined by
\begin{gather}
    \limsup_{k\to\infty}e_k,\quad\text{where}\quad e_k\coloneqq\E[\|\hat{\bm x}_k-\bm x_k^\ast\|].
\end{gather}

We assume the following conditions:

\begin{assumption}\label{ass:additional}
    Let $\gamma\in[0,\|\bm\alpha\|_1^{-1})$ and $\Delta_i\geq0$ ($i=1,2,3$) be constants.
    \begin{enuminasm}
        \item $\|\bm x_k-\hat{\bm x}_k^\ast\|\leq \gamma\|\hat{\bm x}_k-\hat{\bm x}_k^\ast\|$ for all $k\geq 0$.\label{ass:correction}
        \item $\|\E[\hat{\bm x}_k^\ast]-\bm x_k^\ast\|\leq\Delta_1$ and $\E[\|\hat{\bm x}_k^\ast-\bm x_k^\ast\|^2]\leq\Delta_2^2$ for all $k\geq0$.\label{ass:empirical}
        \item $\mynorm{X_k^\ast\bm\alpha-\bm x_k^\ast}\leq\Delta_3$ for all $k\geq0$.\label{ass:prediction}
    \end{enuminasm}
\end{assumption}

\cref{ass:correction} states that the correction step converges to $\hat{\bm x}_k^\ast$ linearly, which holds if the proximal gradient descent method~\cref{eq:proximal_gradient_method} is used \citep[Theorem 10.29]{beck2017first}.

Next, we explain \cref{ass:empirical}.
For $\Delta_1$, if $\hat{\bm x}_k^\ast$ is unbiased, $\Delta_1=0$ holds. 
We can take $\Delta_2=\sigma_1/\mu$ as stated in the next proposition, where $\sigma_1$ and $\mu$ are constants in \cref{ass:problem}.
Furthermore, if we use $b\in\N$ samples to approximate $\hat\ell_k(\bm x)$, the bounds in \cref{ass:empirical} scale as $\Delta_1=O(b^{-1})$ and $\Delta_2=O(b^{-1/2})$ under suitable assumptions.
See \cref{sec:empirical} for the details.

\begin{proposition}\label{prop:empirical0}
    Suppose \cref{ass:problem} holds.
    Then, the following holds for all $k\geq0$:
    \begin{gather}
        \E\bracket*{\|\hat{\bm x}_k^\ast-\bm x_k^\ast\|^2}\leq\frac{\sigma_1^2}{\mu^2}.\label{eq:empirical0}
    \end{gather}
\end{proposition}

\begin{proof}
    Taking $b=1$ in \cref{prop:empirical1} completes the proof.
\end{proof}

\cref{ass:prediction} holds with $\Delta_3=O(h^q)$ if the polynomial basis is used and $\bm x^\ast$ has a bounded $q$-times derivative, as stated in \cref{lem:smooth_solution}.
In general, $\Delta_3=O(h)$ as long as $\mathcal P$ is $\varepsilon$-sensitive, i.e., for all $t,t'\geq0$,
\begin{gather}
    W_1(\mathcal P(t),\mathcal P(t'))\leq\varepsilon|t-t'|.\label{eq:epsilon_sensitivity}
\end{gather}
where $W_1$ is the Wasserstein-1 distance.

\begin{proposition}
    Suppose \cref{ass:existence,ass:smoothness,ass:strong_convexity} hold, $\nabla_{\bm x}\ell(\bm x;\cdot)$ is $L_{1,1}$-Lipschitz continuous, $\inner{\bm1}{\bm\alpha}=1$ holds, and $\mathcal P$ is $\varepsilon$-sensitive.
    Then, for all $k\geq0$, we have
    \begin{align}
        \mynorm*{X_k^\ast\bm\alpha-\bm x_k^\ast}\leq\frac{L_{1,1}}{\mu}\varepsilon h\sum_{j=1}^n\abs{\sum_{i=j}^n\alpha_i}.\label{eq:delta3}
    \end{align}
\end{proposition}

\begin{proof}
    Since $\inner{\bm1}{\bm\alpha}=1$, we have
    \begin{align}
        X_k^\ast\bm\alpha-\bm x_k^\ast
        =\sum_{i=1}^n\alpha_i(\bm x_{k-i}^\ast-\bm x_k^\ast)
        &=\sum_{i=1}^n\sum_{j=1}^i\alpha_i(\bm x_{k-j}^\ast-\bm x_{k-j+1}^\ast)\\
        &=\sum_{j=1}^n\sum_{i=j}^n\alpha_i(\bm x_{k-j}^\ast-\bm x_{k-j+1}^\ast).
    \end{align}
    Taking the norm and using the triangle inequality gives
    \begin{align}
        \mynorm*{X_k^\ast\bm\alpha-\bm x_k^\ast}
        &\leq\sum_{j=1}^n\abs{\sum_{i=j}^n\alpha_i}\|\bm x_{k-j}^\ast-\bm x_{k-j+1}^\ast\|.\label{eq:delta3_1}
    \end{align}
    \cref{prop:drift} yields
    \begin{align}
        \|\bm x_{k-j}^\ast-\bm x_{k-j+1}^\ast\|
        &\leq\frac{L_{1,1}}{\mu}W_1(\mathcal P(t_{k-j}),\mathcal P(t_{k-j+1}))\\
        &\leq\frac{L_{1,1}}{\mu}\varepsilon h,\label{eq:delta3_2}
    \end{align}
    where we used \cref{eq:epsilon_sensitivity} in the last inequality.
    Combining \cref{eq:delta3_1,eq:delta3_2} gives \cref{eq:delta3}.
\end{proof}

The bound on the tracking error is given as follows:

\begin{theorem}\label{thm:tracking_error}
    Suppose \cref{ass:problem,ass:additional} hold.
    Then, there exists $C\geq0$ such that 
    \begin{gather}
        e_k\leq\frac{\|\bm\alpha\|_1\Delta_1+(\gamma\|\bm\alpha\|_1+\|\bm\alpha\|)\Delta_2+\Delta_3}{1-\gamma \|\bm\alpha\|_1}+C(\gamma\|\bm\alpha\|_1)^{k/n}
    \end{gather}
    for all $k\geq0$ and therefore,
    \begin{gather}
        \limsup_{k\to\infty}e_k\leq\frac{\|\bm\alpha\|_1\Delta_1+(\gamma\|\bm\alpha\|_1+\|\bm\alpha\|)\Delta_2+\Delta_3}{1-\gamma \|\bm\alpha\|_1}.\label{eq:tracking_error}
    \end{gather}
\end{theorem}

\begin{proof}
    Let $\hat X_k^\ast\coloneqq\mqty(\hat{\bm x}_{k-1}^\ast&\cdots&\hat{\bm x}_{k-n}^\ast)\in\R^{d\times n}$.
    First, the triangle inequality gives
    \begin{align}
        e_k
        =\E[\|\hat{\bm x}_k-\bm x_k^\ast\|]
        &=\E[\norm{X_k\bm\alpha-\bm x_k^\ast}]\\
        &\leq\E\bracket*{\norm{(X_k-\hat X_k^\ast)\bm\alpha}}
        +\E\bracket*{\norm{(\hat X_k^\ast-X_k^\ast)\bm\alpha}}
        +\norm{X_k^\ast\bm\alpha-\bm x_k^\ast}.\label{eq:tracking_error_1}
    \end{align}
    In the following, we give bounds on each term of the right-hand side.

    The first term can be bounded as follows:
    \begin{align}
        \E\bracket*{\norm{(X_k-\hat X_k^\ast)\bm\alpha}}
        &\leq\sum_{i=1}^n|\alpha_i|\E\bracket*{\|\bm x_{k-i}-\hat{\bm x}_{k-i}^\ast\|}\\
        &\leq\gamma\sum_{i=1}^n|\alpha_i|\E\bracket*{\|\hat{\bm x}_{k-i}-\hat{\bm x}_{k-i}^\ast\|}\\
        &\leq\gamma\sum_{i=1}^n|\alpha_i|\E\bracket*{\|\hat{\bm x}_{k-i}-\bm x_{k-i}^\ast\|}+\gamma\sum_{i=1}^n|\alpha_i|\E\bracket*{\|\bm x_{k-i}^\ast-\hat{\bm x}_{k-i}^\ast\|}\\
        &\leq\gamma\sum_{i=1}^n|\alpha_i|e_{k-i}+\gamma\|\bm\alpha\|_1\Delta_2,
    \end{align}
    where the triangle inequality is used in the first and third inequalities, \cref{ass:correction} is used in the second inequality, and Jensen's inequality and \cref{ass:empirical} are used in the last inequality.
    The second term of the right-hand side in \cref{eq:tracking_error_1} can be bounded as follows:
    \begin{align}
        \E\bracket*{\norm{(\hat X_k^\ast-X_k^\ast)\bm\alpha}}^2
        &\leq\E\bracket*{\norm{(\hat X_k^\ast-X_k^\ast)\bm\alpha}^2}
        =\E\bracket*{\mynorm*{\sum_{i=1}^n\alpha_i(\hat{\bm x}_{k-i}^\ast-\bm x_{k-i}^\ast)}^2}\\
        &=\sum_{i=1}^n\alpha_i^2\E\bracket*{\|\hat{\bm x}_{k-i}^\ast-\bm x_{k-i}^\ast\|^2}+\sum_{i\neq j}\alpha_i\alpha_j\E\bracket*{\inner{\hat{\bm x}_{k-i}^\ast-\bm x_{k-i}^\ast}{\hat{\bm x}_{k-j}^\ast-\bm x_{k-j}^\ast}}\\
        &\leq\sum_{i=1}^n\alpha_i^2\E\bracket*{\|\hat{\bm x}_{k-i}^\ast-\bm x_{k-i}^\ast\|^2}+\sum_{i\neq j}|\alpha_i\alpha_j|\|\E[\hat{\bm x}_{k-i}^\ast]-\bm x_{k-i}^\ast\|\|\E[\hat{\bm x}_{k-j}^\ast]-\bm x_{k-j}^\ast\|\\
        &\leq\|\bm\alpha\|^2\Delta_2^2+\|\bm\alpha\|_1^2\Delta_1^2,
    \end{align}
    where the first inequality follows from Jensen's inequality, the second from the Cauchy--Schwarz inequality and the third from \cref{ass:empirical}.
    Hence, we have
    \begin{align}
        \E\bracket*{\norm{(\hat X_k^\ast-X_k^\ast)\bm\alpha}}
        &\leq\|\bm\alpha\|\Delta_2+\|\bm\alpha\|_1\Delta_1.
    \end{align}
    The third term of the right-hand side in \cref{eq:tracking_error_1} is bounded by $\Delta_3$ from \cref{ass:prediction}.

    Combining all these results, we obtain
    \begin{gather}
        e_k\leq\gamma\sum_{i=1}^n|\alpha_i|e_{k-i}+\gamma\|\bm\alpha\|_1\Delta_2+\|\bm\alpha\|\Delta_2+\|\bm\alpha\|_1\Delta_1+\Delta_3.\label{eq:recurrence_relation2}
    \end{gather}
    Since we have $\gamma\|\bm\alpha\|_1<1$ from \cref{ass:additional}, solving \cref{eq:recurrence_relation2} by using \cref{lem:recursion_lem} completes the proof.
\end{proof}


%% file: 5_alpha.tex
\section{Choices of Coefficients}\label{sec:choices_of_coefficients}
We derived the prediction formula $\hat{\bm x}_k=X_k\bm\alpha$ based on a regression in \cref{sec:proposed}.
However, the tracking error in \cref{thm:tracking_error} is independent of the choice of the basis function $\bm \phi$ (note that its statement and proof do not include $\bm \phi$).
Therefore, choosing $\bm\alpha$ instead of $\bm \phi$ is sufficient.
How to choose the coefficient $\bm\alpha$ is nontrivial, because the best choice depends on the underlying dynamics of the target point $\bm x^\ast(t)$.

\subsection{Tracking Error for Polynomial Regression}\label{sec:tracking_error_polynomial}
One choice is to use the coefficients computed by \eqref{eq:predicted_point} with a predetermined function $\bm \phi\colon\R\to\R^p$.
Here we specialize \cref{thm:tracking_error} to polynomial regression and derive the corresponding tracking-error bound.
We take $\bm\phi(t)=(1,-t,(-t)^2,\ldots,(-t)^{p-1})^\top$.
Then, $\bm\alpha$ can be computed as follows:
\begin{align}
    \alpha_i
    &=(\Phi^\top(\Phi\Phi^\top)^{-1}\bm\phi(0))_i
    =\mqty(1&i&\cdots&i^{p-1})M^{-1}\mqty(1\\0\\\vdots\\0),\label{eq:alpha_matrix}
\end{align}
where $M\in\R^{p\times p}$ is the matrix defined by $M_{l,m}\coloneqq\sum_{j=1}^{n}j^{l+m-2}$ for $l,m=1,2,\ldots,p$.

\begin{example}
    Let $p=2$.
    Then, we have
    \begin{gather}
        \alpha_i=\frac{4n+2-6i}{n(n-1)},\quad i=1,2,\ldots,n,\label{eq:alpha_2}
    \end{gather}
    which are similar to the coefficients used in \citet{maity2022predictor}.
    In this case, we have
    \begin{align}
        \hat{\bm x}_{k}-\hat{\bm x}_{k-1}
        &=\sum_{i=1}^n\alpha_i\bm x_{k-i}-\sum_{i=1}^n\alpha_i\bm x_{k-i-1}\\
        &=\alpha_{1}\bm x_{k-1}+\sum_{i=1}^{n-1}(\alpha_{i+1}-\alpha_i)\bm x_{k-i-1}-\alpha_{n}\bm x_{k-n-1}\\
        &=\frac4n\bm x_{k-1}-\frac{6}{n(n-1)}\bm s_{k-1}+\frac2n\bm x_{k-n-1},
    \end{align}
    where $\bm s_{k}\coloneqq\sum_{i=1}^{n-1}\bm x_{k-i}$.
    This means that storing $\bm s_k$ reduces the computational cost.
    The resulting algorithm is summarized in \cref{alg:proposed_linear}.
    Similarly, for $p\geq3$, we can reduce the computational cost by the same trick.
    Notable point is that the computational cost of the prediction step in \cref{alg:proposed_linear} is independent of $h$, while the existing methods in \cref{tab:error_bound} other than TVSGD require $O(h^{-4/5})$ gradient calls per iteration.
\end{example}

\begin{example}\label{ex:sharp}
    Let $p=n$.
    Then, we have
    \begin{gather}
        \alpha_i=(-1)^{i-1}\binom ni,\quad i=1,2,\ldots,n,
    \end{gather}
    which recovers the coefficients of SHARP~\citep{kamijima2025simple}.
    This is because SHARP is based on the Lagrange interpolation, which corresponds to taking $\Theta_k$ so that the objective function value in \cref{eq:least_squares} becomes 0.
    Hence, SHARP can be interpreted as a special case of the stochastic SHARP.
\end{example}

\begin{algorithm}[tb]
    \caption{Stochastic SHARP (Linear Regression)}
    \label{alg:proposed_linear}
    \begin{algorithmic}[1]
        \Require $\bm x_0\in\R^d$, $n\in\N$
        \State $\bm x_{-n}=\cdots=\bm x_{-1}=\hat{\bm x}_0=\bm x_0$
        \State $\bm s_0=(n-1)\bm x_0$
        \For{$k=1,2,\ldots$}
            \State Update $\hat{\bm x}_k=\hat{\bm x}_{k-1}+\frac4n\bm x_{k-1}-\frac6{n(n-1)}\bm s_{k-1}+\frac2n\bm x_{k-n-1}$
            \Comment {Prediction}\label{line:prediction_linear}
            \State Update $\bm s_k=\bm s_{k-1}+\bm x_{k-1}-\bm x_{k-n}$
            \label{line:update_y}
            \State Set $\bm x_k$ to be an approximate solution of \cref{eq:problem_hat} by e.g., \cref{eq:proximal_gradient_method}
            \Comment {Correction}\label{line:correction_linear}
        \EndFor
    \end{algorithmic}
\end{algorithm}

Next, we compute the leading term of $\|\bm\alpha\|$ and $\|\bm\alpha\|_1$ to see the asymptotic behavior of the tracking error bound in \cref{thm:tracking_error}.

\begin{lemma}
    \label{lem:alpha_norm}
    Let $\bm\alpha\in\R^n$ be the coefficient vector defined in \cref{eq:alpha_matrix}.
    Then, for any fixed $p\geq1$, we have
    \begin{gather}
        \|\bm\alpha\|=O\prn*{\frac1{\sqrt n}}\quad\text{and}\quad\|\bm\alpha\|_1=O\prn*{1}.
    \end{gather}
\end{lemma}

\begin{proof}
    See \cref{sec:proof_of_lem_alpha_norm}.
\end{proof}

Next, we give a sufficient condition for $\Delta_3$ in \cref{ass:additional} to be small.

\begin{lemma}\label{lem:smooth_solution}
    Let $q\in[1,p]$ be an arbitrarily fixed integer and $\bm\alpha\in\R^n$ be the coefficient vector defined in \cref{eq:alpha_matrix}.
    Suppose that $\bm x^\ast(t)$ is $q$-times differentiable and $\sup_{t\geq0}\|(\bm x^\ast)^{(q)}(t)\|<\infty$.
    Then, \cref{ass:prediction} holds with $\Delta_3=\|\bm\alpha\|_1\binom {n}{q}h^{q}\sup_{t\geq0}\|(\bm x^\ast)^{(q)}(t)\|$. 
\end{lemma}

\begin{proof}
    See \cref{sec:proof_of_lem_smooth_solution}.
\end{proof}

Combining \cref{lem:alpha_norm,lem:smooth_solution} with \cref{thm:tracking_error} gives a tracking error bound for the polynomial regression.

\begin{theorem}\label{thm:tracking_error_polynomial}
    Let $q\in[1,p]$ be an arbitrarily fixed integer and $\bm\alpha\in\R^n$ be the coefficient vector defined in \cref{eq:alpha_matrix}.
    Suppose that $\bm x^\ast(t)$ is $q$-times differentiable in $t$ and $\sup_{t\geq0}\|(\bm x^\ast)^{(q)}(t)\|<\infty$, \cref{ass:problem} holds, and \cref{ass:correction,ass:empirical} hold.
    Set $n=\Theta(h^{-2q/(2q+1)})$.
    Then, we have
    \begin{gather}
        \limsup_{k\to\infty}e_k\leq O\prn*{\Delta_1+\gamma\Delta_2+(\Delta_2+1)h^{q/(2q+1)}}.\label{eq:tracking_error_polynomial}
    \end{gather}
\end{theorem}

\begin{proof}
    \cref{thm:tracking_error,lem:smooth_solution} give
    \begin{align}
        \limsup_{k\to\infty}e_k
        &\leq\frac{1}{1-\gamma\|\bm\alpha\|_1}\prn*{\|\bm\alpha\|_1\Delta_1+(\gamma\|\bm\alpha\|_1+\|\bm\alpha\|)\Delta_2+\|\bm\alpha\|_1\binom {n}{q}h^{q}\sup_{t\geq0}\|(\bm x^\ast)^{(q)}(t)\|}\\
        &=O\prn*{\Delta_1+\prn*{\gamma+\frac1{\sqrt n}}\Delta_2+n^qh^{q}},
    \end{align}
    where \cref{lem:alpha_norm} and $\binom {n}{q}\leq \frac{n^q}{q!}$ are used in the equality.
    Using $n=\Theta(h^{-2q/(2q+1)})$ completes the proof.
\end{proof}

\Cref{thm:tracking_error_polynomial} implies that, if $\Delta_1=0$ (i.e., $\hat{\bm x}_k^\ast$ is unbiased) and $\gamma=O(h^{q/(2q+1)})$, the asymptotic tracking error is $O(h^{q/(2q+1)})$, which is smaller than the $O(h^{1/3})$ tracking error bound of TVSGD~\citep{cutler2021stochastic} when $q\geq2$.
When $h$ is sufficiently small, taking larger $p$, which is the maximum of $q$, leads better performance because the tracking error bound $O(h^{p/(2p+1)})$ becomes smaller.
However, it is not always true when $h$ is not sufficiently small because $\|\bm\alpha\|$ and $\|\bm\alpha\|_1$ depend on $p$ (see \cref{sec:proof_of_lem_alpha_norm} for details).

Even if $\Delta_1>0$, using a mini-batch of size $b$ to approximate $\hat\ell_k(x)$ can reduce $\Delta_1$ as shown in \cref{sec:empirical}.
In this case, taking $b=\Theta(h^{-r})$ for some $r>0$ and $n=\Theta(b^{-1/(2q+1)}h^{-2q/(2q+1)})$ gives the $O(\gamma h^{r/2}+h^{(r+1)q/(2q+1)})$ tracking error bound, where we used \cref{cor:empirical2} (cf. the same batch strategy to TVSGD~\citep{cutler2021stochastic} gives the $O(h^{(1+r)/3})$ tracking error bound).
This means that the tracking error can be made arbitrarily small by increasing $b$ at the cost of the computational complexity of the correction step.

\subsection{Online Learning-based Coefficients}\label{sec:online_learning}
Another choice is to start with an initial coefficient $\bm\alpha^1\in\R^n$ and update it iteratively by online learning.
This type of learning-based parameter tuning has been proposed in the study of non-time-varying optimization problems \citep{rubio2017convergence,cutkosky2023optimal,gao2025gradient,chu2025gradient,chu2025provable,jiang2025online,jiang2025improved}.
Thus, in time-varying settings, this approach is expected to find a good coefficient $\bm\alpha$ adapted to the underlying dynamics of the target point $\bm x^\ast(t)$ without prior knowledge, although the theoretical analysis is left for future work.

The prediction step can be written as $\hat{\bm x}_k=X_{k}\bm\alpha^k$.
Note that, given $\bm\alpha^k$, we can compute $\hat{\bm x}_k$ without $\phi_1(t),\ldots,\phi_p(t)$.
This prediction incurs a loss of $\hat\ell_k(X_k\bm\alpha^k)+\psi(X_k\bm\alpha^k)$ and the coefficient $\bm\alpha^k$ is updated by using an online optimization method.
For example, we can use the online gradient descent~\citep{zinkevich2003online}:
\begin{gather}
    \bm\alpha^{k+1}
    =\proj_{\mathcal A}\prn*{\bm\alpha^k-\eta_kX_k^\top\prn*{\nabla\hat\ell_k(X_k\bm\alpha^k)+\bm g^k}},\quad \bm g^k\in\partial \psi(X_k\bm\alpha^k)\label{eq:subgradient_descent}
\end{gather}
where $\eta_k>0$ is a step size and $\mathcal A\subset\R^n$ is a compact convex set defined by
\begin{gather}
    \mathcal A\coloneqq\Set{\bm\alpha\in\R^n}{\inner{\bm1}{\bm\alpha}=1,~\|\bm\alpha\|_2\leq D},\label{eq:constraint_set}
\end{gather}
for some $D\geq n^{-1/2}$.
We restrict the coefficient $\bm\alpha^k$ to a compact convex set $\mathcal A\subset\R^n$ because most online optimization algorithms require such a constraint to guarantee the theoretical performance.
The condition $\inner{\bm1}{\bm\alpha}=1$ ensures that the predicted point $\hat{\bm x}_k$ does not depend on the scaling and the translation of $x_{k-1},\ldots,x_{k-n}$.
This naturally holds if the coefficient vector $\bm\alpha^k$ is computed by \eqref{eq:predicted_point} with a given basis function $\bm\phi$ satisfying $\phi_1(t)=1$ as explained in \cref{sec:proposed}.
The condition $\|\bm\alpha\|_2\leq D$ is to ensure the boundedness of the feasible region.
Since a good $\bm\alpha$ must be included in $\mathcal A$, taking $D=1$ is sufficient considering \cref{lem:alpha_norm}.
The projection onto $\mathcal A$ can be computed explicitly.

We can also use the Ader algorithm \citep{zhang2018adaptive}, which enjoys a better dynamic regret bound than the online gradient descent.
\color{black}

%% file: 6_numerical.tex
\section{Numerical Experiment}\label{sec:numerical}
This section presents numerical experiments to demonstrate the performance of the proposed method.
All the experiments were conducted in Python 3.13.5 on a MacBook Air whose processor is Apple M4 and memory is 32 GB.

We consider the following time-varying LASSO problem:
\begin{align}
    \min_{\bm x\in\R^d}\E_{\bm y\sim\mathcal N(A\bm x^\ast(t),I_{d_y})}\bracket*{\frac12\|A\bm x-\bm y\|^2}+\lambda\|\bm x\|_1.
\end{align}
In the experiment, we take $d=5$, $n=10$, and randomly generate $A\in\R^{d_y\times d}$ via its singular value decomposition so that its minimal singular value is $\sqrt\mu$ and maximum is $\sqrt{L}$.
Hence, the first term of the objective function is $\mu$-strongly convex and $L$-smooth.
A time-varying target vector $\bm x^\ast(t)\in\R^d$ is defined by
\begin{align}
    x_i^\ast(t)\coloneqq\sin\prn*{t+\frac{2\pi(i-1)}d}+\frac12\sin\prn*{2t+\frac{2\pi(i-1)}d}
\end{align}
for $i=1,\ldots,d$.
We set $\bm x_0=(2,0,\ldots,0)^\top\in\R^d$, $h\in\{0.0005,0.001,0.002,0.005,0.01,0.02,0.05,0.1\}$, $L=h^{-2/3}/2$, $\mu=L/10$, and $\lambda\in\{0,1\}$.
We compare the average of $\|\hat{\bm x}_k-\bm x_k^\ast\|$ over all $k$ satisfying $4\leq t_k<5$ for each $h$.

The algorithms we compare are the following:
\begin{itemize}
    \item \textbf{Time-Varying Proximal Stochastic Gradient Descent (TVSGD)}. 
    The algorithm proposed by \citet{dixit2019online} and \citet{cutler2021stochastic}.
    Its step size is set to $h^{2/3}$ to guarantee better tracking performance.
    The estimator of the gradient is computed by sampling $y$.
    \item \textbf{Stochastic Predictor-Corrector based method (Stochastic PC)}.
    The algorithm proposed by \citet{maity2022predictor}.
    The step size is set to $h^{4/5}$.
    We compute the derivative estimators following the method described in the original paper.
    This method is only applicable to the case where the loss function is smooth, i.e., $\lambda=0$.
    \item \textbf{Static Contractive Filter (TV Contract)}.
    The algorithm proposed by \citet{simonetto2024nonlinear}.
    The parameters are set to $\rho=0.5$, $\Delta=0.1$, $\omega_1=\omega_2=0.9$, $Q=200I_d$, $R=50I_d$, and $P=C=5$ (the notation is the same as the original paper).
    \item \textbf{Stochastic SHARP method (Proposed)}.
    \cref{alg:proposed} with $\psi=0$.
    The correction step is computed by \cref{eq:proximal_gradient_method} with $\beta=1/L$.
    The coefficient $\bm \alpha$ is computed in three ways: one is by \cref{eq:alpha_2} with $n=\floor{h^{-4/5}}$ ($p=2$), another is by \cref{eq:alpha_matrix} with $(n,p)=(\floor{h^{-6/7}},3)$, and the other is by \cref{eq:subgradient_descent} with $n=\floor{h^{-1}}$, $\bm \alpha^0=\bm1/n$, $D=1$, and $\eta_k=0.01\cdot2^{-\floor{\log_2k}/2}$ (Online).
    This step size selection is based on the doubling trick \citep[Section 4.6]{cesa1997use}.
    The number of the correction step is set to $c_{\max}=30$.
\end{itemize}

\subsection{Experiment 1: No Regularization Case}
The results for $\lambda=0$ are shown in \cref{fig:no_regularization}.
The proposed method with $p=2$ achieves a similar performance to the Stochastic PC method in small $h$, while the computational cost is much lower.
The proposed method with $p=3$ also achieves good performance, but it is slightly worse than that with $p=2$.
This is because a large $p$ makes $\|\bm \alpha\|$ and $\|\bm \alpha\|_1$ large, which leads to a larger upper bound on the theoretical guarantee as stated in \cref{sec:tracking_error_polynomial}.
The proposed method with the online learning of the coefficient achieves good performance in some $h$, although it is unstable in smaller $h$.


\begin{figure}[tb]
    \centering
    \includegraphics[width=0.6\textwidth]{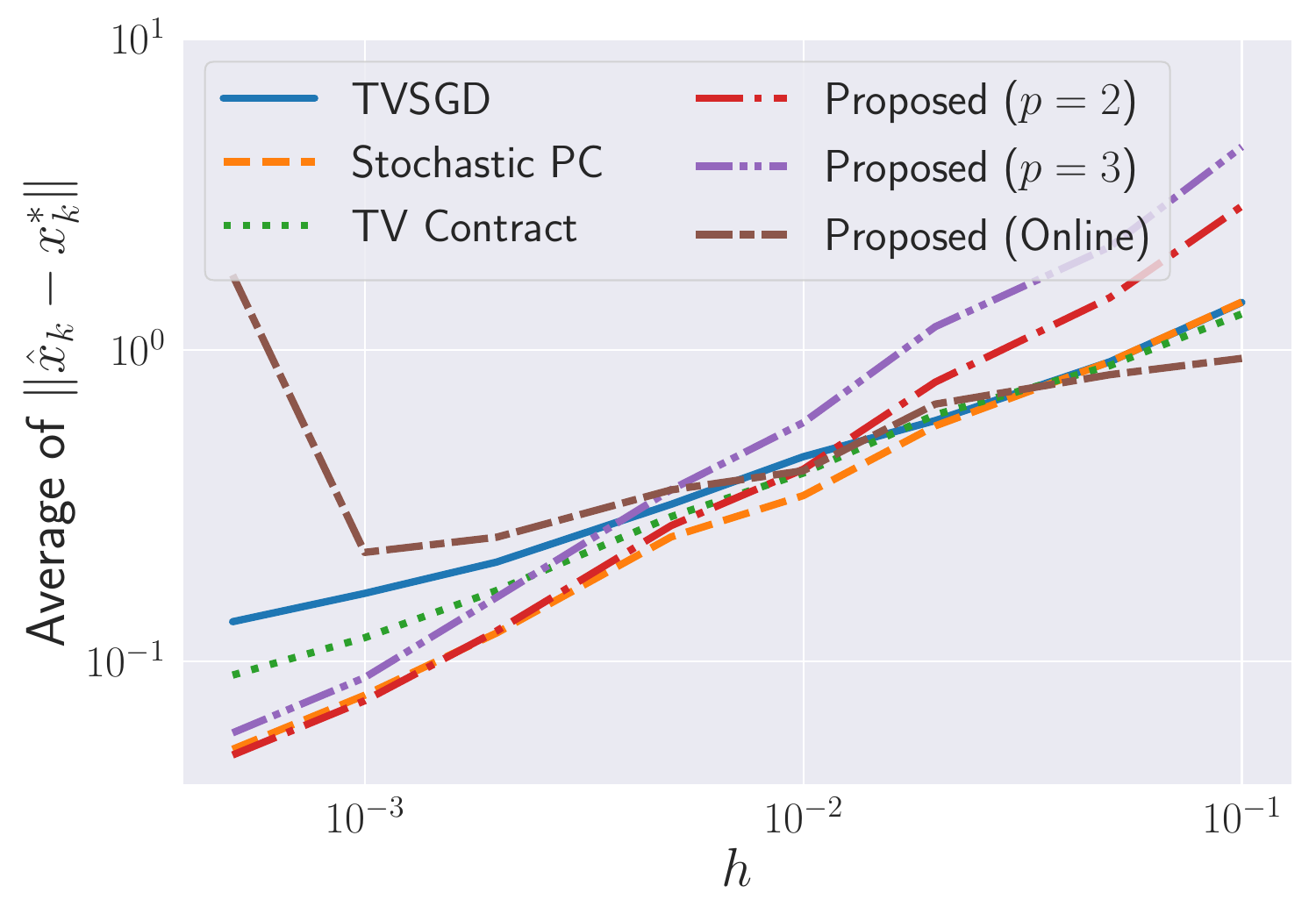}
    \caption{Results of Experiment 1. The average of $\|\hat{\bm x}_k-\bm x_k^\ast\|$ over all $k$ satisfying $4\leq t_k<5$.}
    \label{fig:no_regularization}
\end{figure}


\subsection{Experiment 2: Regularization Case}
The results for $\lambda=1$ are shown in \cref{fig:regularization}.
The proposed method with $p=2,3$ outperforms existing methods in small $h$, and especially $p=2$ achieves the best performance among all the methods.
The online learning-based method performs well in some $h$, but it is unstable in smaller $h$ as in Experiment~1.


\begin{figure}[tb]
    \centering
    \includegraphics[width=0.6\textwidth]{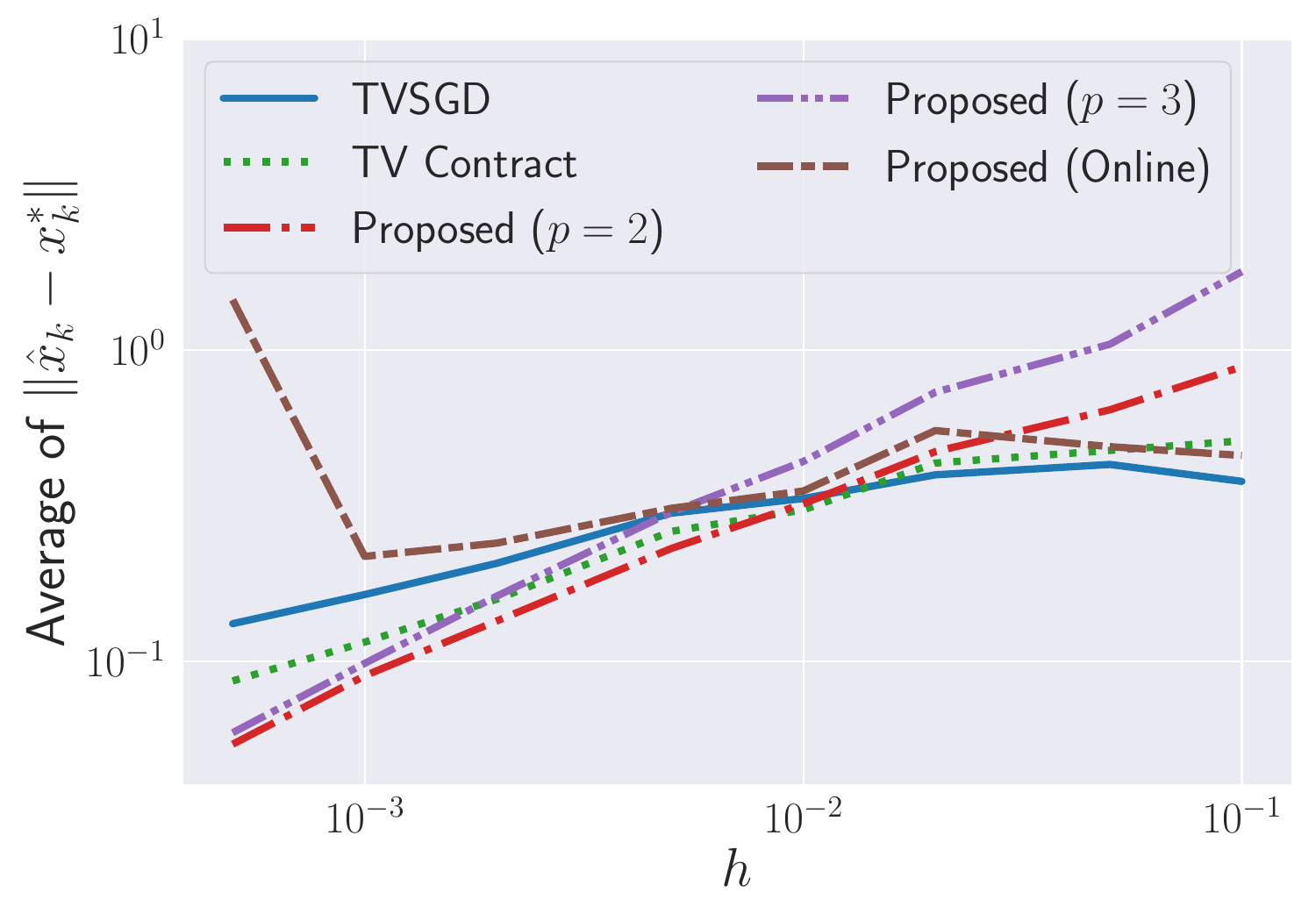}
    \caption{Results of Experiment 2. The average of $\|\hat{\bm x}_k-\bm x_k^\ast\|$ over all $k$ satisfying $4\leq t_k<5$.}
    \label{fig:regularization}
\end{figure}


%% file: 7_conclusion.tex
\section{Conclusion}\label{sec:conclusion}
We proposed a new method for solving stochastic time-varying optimization problems.
The proposed method is based on the prediction-correction framework and uses a prediction step based on regression, which yields a prediction formula that can be computed by a linear combination of the previous solutions.
We derived a theoretical guarantee for the tracking error of the proposed method.
The numerical experiments showed that the proposed method achieves a similar performance to the existing methods while being computationally efficient.

%% file: 9_appendix.tex
\section{Omitted Proofs in \cref{sec:tracking_error}}
\subsection{Technical Lemmas}\label{sec:technical_lemmas}
This section introduces some technical lemmas.
First, we provide some properties of strongly convex functions and smooth functions.
For the proof, see \citet{nesterov2018lectures}, for example.

\begin{lemma}\label{lem:strong_convexity}
    Suppose that $\varphi\colon\R^d\to\R$ is $\mu$-strongly convex.
    Then, for all $\bm x,\bm x'\in\R^d$, 
    \begin{gather}
        \inner{\nabla\varphi(\bm x)-\nabla\varphi(\bm x')}{\bm x-\bm x'}\geq\mu\|\bm x-\bm x'\|^2\quad\text{and}\label{eq:strong_convexity2}\\
        \nabla^2\varphi(\bm x)\succeq \mu I.\label{eq:strong_convexity3}
    \end{gather}
\end{lemma}

\begin{lemma}\label{lem:smooth}
    Suppose that $\varphi\colon\R^d\to\R$ is $L_{2,0}$-smooth.
    Then, for all $\bm x,\bm x'\in\R^d$, 
    \begin{gather}
        \|\nabla\varphi(\bm x)-\nabla\varphi(\bm x')\|\leq L_{2,0}\|\bm x-\bm x'\|.\label{eq:smooth2}
    \end{gather}
\end{lemma}


Next, we prepare an inequality that minimizers of convex functions satisfy.

\begin{lemma}
    \label{lem:optimality_property}
    Suppose that $\varphi,\tilde\varphi\colon\R^d\to\R$ are differentiable convex functions and $\psi\colon\R^d\to\R$ is a convex function.
    Let $\bm x^\ast$ and $\tilde {\bm x}^\ast$ be minimizers of $\varphi+\psi$ and $\tilde\varphi+\psi$, respectively.
    Then, we have
    \begin{gather}
        \inner{\nabla\varphi(\bm x^\ast)-\nabla\tilde\varphi(\tilde {\bm x}^\ast)}{\bm x^\ast-\tilde {\bm x}^\ast}\leq0.\label{eq:optimality_property}
    \end{gather}
\end{lemma}

\begin{proof}
    The optimality condition gives
    \begin{align}
        \begin{dcases*}
            -\nabla\varphi(\bm x^\ast)\in\partial\psi(\bm x^\ast),\\
            -\nabla\tilde\varphi(\tilde {\bm x}^\ast)\in\partial\psi(\tilde {\bm x}^\ast).
        \end{dcases*}
    \end{align}
    The definition of the subdifferential $\partial\psi$ yields
    \begin{align}
        \begin{dcases*}
            \psi(\tilde {\bm x}^\ast)\geq\psi(\bm x^\ast)-\inner{\nabla\varphi(\bm x^\ast)}{\tilde {\bm x}^\ast-\bm x^\ast},\\
            \psi(\bm x^\ast)\geq\psi(\tilde {\bm x}^\ast)-\inner{\nabla\tilde\varphi(\tilde {\bm x}^\ast)}{\bm x^\ast-\tilde {\bm x}^\ast}.
        \end{dcases*}
    \end{align}
    Summing up these inequalities completes the proof.
\end{proof}

Next, we justify the interchange of the expectation and the gradient.

\begin{lemma}\label{lem:gradient_expectation}
    Suppose that \cref{ass:existence,ass:smoothness} hold.
    Then, we have
    \begin{gather}
        \nabla\E[\ell(\bm x;\bm y)]=\E[\nabla_{\bm x}\ell(\bm x;\bm y)]\label{eq:gradient_expectation}
    \end{gather}
    for all $\bm x\in\R^d$.
\end{lemma}

\begin{proof}
    For all $\bm v\in\R^d$ satisfying $\|\bm v\|=1$ and for all $h>0$, we have
    \begin{align}
        \abs{\frac{\ell(\bm x+h\bm v;\bm y)-\ell(\bm x;\bm y)}h}
        &=\abs{\int_0^1\inner{\nabla_{\bm x}\ell(\bm x+th\bm v;\bm y)}{\bm v}\dd t}\\
        &\leq\int_0^1\mynorm*{\nabla_{\bm x}\ell(\bm x+th\bm v;\bm y)}\dd t\\
        &\leq\mynorm*{\nabla_{\bm x}\ell(\bm x;\bm y)}+\int_0^1\mynorm*{\nabla_{\bm x}\ell(\bm x+th\bm v;\bm y)-\nabla_{\bm x}\ell(\bm x;\bm y)}\dd t\\
        &\leq\mynorm*{\nabla_{\bm x}\ell(\bm x;\bm y)}+L_{2,0}\int_0^1th\dd t\\
        &=\mynorm*{\nabla_{\bm x}\ell(\bm x;\bm y)}+\frac{L_{2,0}h}2,
    \end{align}
    where the first equality follows from the fundamental theorem of calculus, the first inequality from the Cauchy--Schwarz inequality, and the third inequality follows from \cref{ass:smoothness,eq:smooth2}.
    Since $\mynorm*{\nabla_{\bm x}\ell(\bm x;\bm y)}+\frac12L_{2,0}h$ is integrable in $\bm y$ by \cref{ass:existence}, the dominated convergence theorem guarantees 
    \begin{align}
        \inner{\nabla\E[\ell(\bm x;\bm y)]}{\bm v}
        &=\lim_{h\to0}\E\bracket*{\frac{\ell(\bm x+h\bm v;\bm y)-\ell(\bm x;\bm y)}h}\\
        &=\E\bracket*{\lim_{h\to0}\frac{\ell(\bm x+h\bm v;\bm y)-\ell(\bm x;\bm y)}h}\\
        &=\inner{\E[\nabla_{\bm x}\ell(\bm x;\bm y)]}{\bm v}
    \end{align}
    for all $\bm v\in\R^d$ satisfying $\|\bm v\|=1$, which implies \cref{eq:gradient_expectation}.
\end{proof}

Repeating the same discussion on each component of $\nabla_{\bm x}\ell(\bm x;\bm y)$, we get the following corollary:

\begin{corollary}\label{cor:hessian_expectation}
    Suppose that \cref{ass:existence} holds, $\E[\nabla_{\bm x\bm x}\ell(\bm x;\bm y)]$ exists for all $\bm x\in\R^d$, and $\nabla_{\bm x\bm x}\ell(\bm x;\bm y)$ is $L_{3,0}$-Lipschitz continuous in $\bm x$.
    Then, we have
    \begin{gather}
        \nabla^2\E[\ell(\bm x;\bm y)]=\E[\nabla_{\bm x\bm x}\ell(\bm x;\bm y)]\label{eq:hessian_expectation}
    \end{gather}
    for all $\bm x\in\R^d$.
\end{corollary}

\subsection{Relation between True Solution and Empirical Solution}\label{sec:empirical}
This subsection presents relations between the true solution $\bm x_k^\ast=\argmin_{\bm x}f(\bm x;t_k)$ and the empirical solution $\hat{\bm x}_k^\ast=\argmin_{\bm x}\hat f_k(\bm x)$ when 
\begin{align}
    \hat\ell_k(\bm x)=\frac1b\sum_{i=1}^b\ell(\bm x;\bm y_{k,i}),
\end{align}
where $\bm y_{k,1},\ldots,\bm y_{k,b}$ are i.i.d. samples from the data distribution.

\begin{proposition}\label{prop:empirical1}
    Suppose \cref{ass:existence,ass:gradient_variance,ass:strong_convexity,ass:smoothness}.
    Then, the following holds for all $k\geq0$:
    \begin{gather}
        \E\bracket*{\|\hat{\bm x}_k^\ast-\bm x_k^\ast\|^2}\leq\frac{\sigma_1^2}{\mu^2b}.\label{eq:empirical1}
    \end{gather}
\end{proposition}

\begin{proof}
    \Cref{lem:optimality_property} with $\varphi=\E[\ell(\cdot;\bm y)]$ and $\tilde\varphi=\hat\ell_k$ gives
    \begin{gather}
        \inner{\nabla\hat \ell_k(\hat{\bm x}_k^\ast)}{\bm x_k^\ast-\hat{\bm x}_k^\ast}\geq\inner{\E[\nabla\ell(\bm x_k^\ast;\bm y)]}{\bm x_k^\ast-\hat{\bm x}_k^\ast},\label{eq:empirical1_2}
    \end{gather}
    where \cref{lem:gradient_expectation} is used.
    By combining this inequality with \cref{eq:strong_convexity2} for $\varphi=\hat\ell_k$, we obtain
    \begin{align}
        \mu\|\bm x_k^\ast-\hat{\bm x}_k^\ast\|^2
        &\leq\inner{\nabla\hat\ell_k(\bm x_k^\ast)-\nabla\hat\ell_k(\hat{\bm x}_k^\ast)}{\bm x_k^\ast-\hat{\bm x}_k^\ast}\\
        &\leq\inner{\nabla\hat\ell_k(\bm x_k^\ast)-\E[\nabla\ell(\bm x_k^\ast;\bm y)]}{\bm x_k^\ast-\hat{\bm x}_k^\ast}\\
        &\leq\|\nabla\hat\ell_k(\bm x_k^\ast)-\E[\nabla\ell(\bm x_k^\ast;\bm y)]\|\|\bm x_k^\ast-\hat{\bm x}_k^\ast\|,
    \end{align}
    where the first inequality follows from \cref{eq:strong_convexity2}, the second inequality from \cref{eq:empirical1_2}, and the last inequality follows from the Cauchy--Schwarz inequality.
    Therefore, we have
    \begin{gather}
        \mu^2\|\bm x_k^\ast-\hat{\bm x}_k^\ast\|^2\leq\|\nabla\hat\ell_k(\bm x_k^\ast)-\E[\nabla\ell(\bm x_k^\ast;\bm y)]\|^2.
    \end{gather}
    Taking the expectation with respect to the samples $\bm y_{k,1},\ldots,\bm y_{k,b}$ gives
    \begin{align}
        \mu^2\E\bracket*{\|\bm x_k^\ast-\hat{\bm x}_k^\ast\|^2}
        &\leq \E\bracket*{\|\nabla\hat\ell_k(\bm x_k^\ast)-\E[\nabla\ell(\bm x_k^\ast;\bm y)]\|^2}\\
        &=\E\bracket*{\mynorm*{\frac1b\sum_{i=1}^b\nabla \ell(\bm x_k^\ast;\bm y_{k,i})-\E[\nabla\ell(\bm x_k^\ast;\bm y)]}^2}\\
        &=\frac1{b^2}\E\bracket*{\mynorm*{\sum_{i=1}^b(\nabla \ell(\bm x_k^\ast;\bm y_{k,i})-\E[\nabla \ell(\bm x_k^\ast;\bm y)])}^2}\\
        &=\frac1{b^2}\sum_{i=1}^b\E\bracket*{\mynorm*{\nabla \ell(\bm x_k^\ast;\bm y_{k,i})-\E[\nabla \ell(\bm x_k^\ast;\bm y)]}^2}\\
        &\leq\frac{\sigma_1^2}{b},\label{eq:empirical1_1}
    \end{align}
    where the third equality follows from the independence of the samples $\bm y_{k,1},\ldots,\bm y_{k,b}$.
\end{proof}

Applying Jensen's inequality to \cref{prop:empirical1} gives the following corollary:

\begin{corollary}\label{cor:empirical2}
    Under the same assumption as \cref{prop:empirical1}, the following holds for all $k\geq0$:
    \begin{gather}
        \|\E\bracket*{\hat {\bm x}_k^\ast}-\bm x_k^\ast\|\leq\frac{\sigma_1}{\mu\sqrt b}.\label{eq:empirical2}
    \end{gather}
\end{corollary}

When $\psi=0$, a tighter bound than \cref{cor:empirical2} can be obtained.

\begin{proposition}\label{prop:empirical3}
    Suppose that $\psi=0$, \cref{ass:existence,ass:gradient_variance,ass:strong_convexity,ass:smoothness} hold, $\E[\|\nabla_{\bm x\bm x}\ell(\bm x;\bm y)\|]$ exist for all $\bm x\in\R^d$, $\nabla_{\bm x\bm x}\ell(\bm x;\bm y)$ is $L_{3,0}$-Lipschitz continuous in $\bm x$, and there exists $\sigma_2\geq0$ such that
    \begin{align}
        \E\bracket*{\mynorm*{\nabla_{\bm x\bm x} \ell(\bm x_k^\ast;\bm y)-\E[\nabla_{\bm x\bm x} \ell(\bm x_k^\ast;\bm y)]}^2}&\leq\sigma_2^2        
    \end{align}
    for all $k\geq0$.
    Then, the following holds for all $k\geq0$:
    \begin{gather}
        \|\E[\hat {\bm x}_k^\ast]-\bm x_k^\ast\|\leq \prn*{\frac{L_{3,0}\sigma_1^2}{2\mu^3}+\frac{\sigma_1\sigma_2}{\mu^2}}\frac1b\label{eq:empirical3}
    \end{gather}
\end{proposition}

\begin{proof}
    The Lipschitz continuity of $\nabla_{\bm x\bm x}\ell(\bm x;\bm y)$ gives
    \begin{gather}
        \|\nabla^2\hat\ell_k(\bm x_k^\ast)(\hat {\bm x}_k^\ast-\bm x_k^\ast)+\nabla\hat\ell_k(\bm x_k^\ast)-\nabla\hat\ell_k(\hat {\bm x}_k^\ast)\|\leq\frac{L_{3,0}}2\|\hat {\bm x}_k^\ast-\bm x_k^\ast\|^2
    \end{gather}
    \citep[Lemma 4.1.1]{nesterov2018lectures}.
    From the optimality condition $\nabla\hat\ell_k(\hat {\bm x}_k^\ast)=0$ and \cref{eq:strong_convexity3} with $\varphi=\hat\ell_k$, we have
    \begin{align}
        \|\hat {\bm x}_k^\ast-\bm x_k^\ast+\nabla^2\hat\ell_k(\bm x_k^\ast)^{-1}\nabla\hat\ell_k(\bm x_k^\ast)\|
        &\leq\|\nabla^2\hat\ell_k(\bm x_k^\ast)^{-1}\|\|\nabla^2\hat\ell_k(\bm x_k^\ast)(\hat {\bm x}_k^\ast-\bm x_k^\ast)+\nabla\hat\ell_k(\bm x_k^\ast)\|\\
        &\leq\frac{L_{3,0}}{2\mu}\|\hat {\bm x}_k^\ast-\bm x_k^\ast\|^2.\label{eq:empirical3_1}
    \end{align}
    Therefore, we get 
    \begin{align}
        \|\E[\hat {\bm x}_k^\ast]-\bm x_k^\ast\|
        &\leq\|\E[\hat {\bm x}_k^\ast-\bm x_k^\ast+\nabla^2\hat\ell_k(\bm x_k^\ast)^{-1}\nabla\hat\ell_k(\bm x_k^\ast)]\|+\|\E[\nabla^2\hat\ell_k(\bm x_k^\ast)^{-1}\nabla\hat\ell_k(\bm x_k^\ast)]\|\\
        &\leq \E[\|\hat {\bm x}_k^\ast-\bm x_k^\ast+\nabla^2\hat\ell_k(\bm x_k^\ast)^{-1}\nabla\hat\ell_k(\bm x_k^\ast)\|]+\|\E[\nabla^2\hat\ell_k(\bm x_k^\ast)^{-1}\nabla\hat\ell_k(\bm x_k^\ast)]\|\\
        &\leq \frac{L_{3,0}}{2\mu}\E[\|\hat {\bm x}_k^\ast-\bm x_k^\ast\|^2]+\|\E[\nabla^2\hat\ell_k(\bm x_k^\ast)^{-1}\nabla\hat\ell_k(\bm x_k^\ast)]\|\\
        &\leq \frac{L_{3,0}\sigma_1^2}{2\mu^3b}+\|\E[\nabla^2\hat\ell_k(\bm x_k^\ast)^{-1}\nabla\hat\ell_k(\bm x_k^\ast)]\|,
    \end{align}
    where the first inequality follows from the triangle inequality, the second from Jensen's inequality, the third from \cref{eq:empirical3_1}, and the last from \cref{prop:empirical1}.

    To finish the proof, we bound the second term.
    Let $H\coloneqq\nabla^2f(\bm x_k^\ast)$, $\hat H\coloneqq\nabla^2\hat\ell_k(\bm x_k^\ast)$, $\Delta_H\coloneqq \hat H-H$, and $\hat {\bm g}\coloneqq\nabla \hat\ell_k(\bm x_k^\ast)$.
    Then, we have
    \begin{align}
        \hat H^{-1}
        &=H^{-1}-H^{-1}\hat H\hat H^{-1}+\hat H^{-1}\\
        &=H^{-1}-H^{-1}(\hat H-H)\hat H^{-1}\\
        &=H^{-1}-H^{-1}\Delta_H\hat H^{-1}.
    \end{align}
    Hence, applying the triangle inequality gives
    \begin{align}
        \|\E[\nabla^2\hat\ell_k(\bm x_k^\ast)^{-1}\nabla\hat\ell_k(\bm x_k^\ast)]\|
        =\|\E[\hat H^{-1}\hat {\bm g}]\|
        &\leq\|\E[H^{-1}\hat {\bm g}]\|+\|\E[H^{-1}\Delta_H\hat H^{-1}\hat {\bm g}]\|\\
        &=\|\E[H^{-1}\Delta_H\hat H^{-1}\hat {\bm g}]\|,
    \end{align}
    where we used $\E[\hat {\bm g}]=0$ in the last equality.
    This can be bounded as follows:
    \begin{align}
        \|\E[H^{-1}\Delta_H\hat H^{-1}\hat {\bm g}]\|
        &\leq\E[\|H^{-1}\Delta_H\hat H^{-1}\hat {\bm g}\|]\\
        &\leq\frac1{\mu^2}\E[\|\Delta_H\|\|\hat {\bm g}\|]\\
        &\leq\frac1{\mu^2}\sqrt{\E[\|\Delta_H\|^2]}\sqrt{\E[\|\hat {\bm g}\|^2]},
    \end{align}
    where the first inequality follows from Jensen's inequality, the second from \cref{eq:strong_convexity3} with $\varphi=f,\hat\ell_k$, and the last from the Cauchy--Schwarz inequality.
    By the similar discussion as \cref{eq:empirical1_1}, we have $\E[\|\hat {\bm g}\|^2]\leq\sigma_1^2/b$ and $\E[\|\Delta_H\|^2]\leq\sigma_2^2/b$.
    To get the latter bound, we used \cref{cor:hessian_expectation}.

    By combining these inequalities, we obtain
    \begin{align}
        \|\E[\hat {\bm x}_k^\ast]-{\bm x}_k^\ast\|
        &\leq \prn*{\frac{L_{3,0}\sigma_1^2}{2\mu^3}+\frac{\sigma_1\sigma_2}{\mu^2}}\frac1b,
    \end{align}
    which completes the proof.
\end{proof}

\subsection{Drift of the Optimal Solution}\label{sec:drift}
This subsection gives a bound on the drift of the optimal solution $\bm x^\ast(t)$ by using the Wasserstein-1 distance.
The idea is based on \citet[Theorem 3.5 (a)]{perdomo2020performative}.

\begin{proposition}\label{prop:drift}
    Suppose \cref{ass:existence,ass:smoothness,ass:strong_convexity} and that $\nabla_{\bm x}\ell(\bm x;\cdot)$ is $L_{1,1}$-Lipschitz continuous.
    Then, for all $k,k'\geq0$, we have
    \[
        \|\bm x_k^\ast-\bm x_{k'}^\ast\|\leq\frac{L_{1,1}}{\mu}W_1(\mathcal P_k,\mathcal P_{k'}),
    \]
    where $\mathcal P_k\coloneqq\mathcal P(t_k)$.
\end{proposition}

\begin{proof}
    \Cref{eq:strong_convexity2} gives
    \begin{gather}
        \mu\|\bm x_k^\ast-\bm x_{k'}^\ast\|^2
        \leq\inner{\nabla_{\bm x} \ell(\bm x_k^\ast;\bm y)-\nabla_{\bm x} \ell(\bm x_{k'}^\ast;\bm y)}{\bm x_k^\ast-\bm x_{k'}^\ast}.
    \end{gather}
    Taking the expectation with respect to $\bm y\sim \mathcal P_k$ yields
    \begin{align}
        \mu\|\bm x_k^\ast-\bm x_{k'}^\ast\|^2
        &\leq\inner{\E_{\bm y\sim \mathcal P_k}[\nabla_{\bm x} \ell(\bm x_k^\ast;\bm y)]-\E_{\bm y\sim \mathcal P_k}[\nabla_{\bm x} \ell(\bm x_{k'}^\ast;\bm y)]}{\bm x_k^\ast-\bm x_{k'}^\ast}.\label{eq:drift1}
    \end{align}

    Let $\tilde\ell(\bm y)\coloneqq\inner{\nabla_{\bm x} \ell(\bm x_{k'}^\ast;\bm y)}{\bm x_k^\ast-\bm x_{k'}^\ast}$.
    Then, we have
    \begin{align}
        \abs{\tilde\ell(\bm y)-\tilde\ell(\bm y')}
        &=\abs{\inner{\nabla_{\bm x} \ell(\bm x_{k'}^\ast;\bm y)-\nabla_{\bm x} \ell(\bm x_{k'}^\ast;\bm y')}{\bm x_k^\ast-\bm x_{k'}^\ast}}\\
        &\leq L_{1,1}\norm{\bm y-\bm y'}\norm{\bm x_k^\ast-\bm x_{k'}^\ast},
    \end{align}
    where the Cauchy--Schwarz inequality and $L_{1,1}$-Lipschitz continuity of $\nabla_{\bm x}\ell(\bm x;\cdot)$ is used, which implies $\tilde\ell$ is $L_{1,1}\|\bm x_k^\ast-\bm x_{k'}^\ast\|$-Lipschitz continuous.
    Kantorovich--Rubinstein duality \citep[Particular Case 5.16]{villani2008optimal} gives
    \begin{align}
        L_{1,1}\|\bm x_k^\ast-\bm x_{k'}^\ast\|W_1(\mathcal P_k,\mathcal P_{k'})
        &=L_{1,1}\|\bm x_k^\ast-\bm x_{k'}^\ast\|\sup_{\varphi\in\mathrm{Lip}_1(\R^{d_y})}\set*{\E_{\bm y\sim \mathcal P_{k'}}[\varphi(\bm y)]-\E_{\bm y\sim \mathcal P_k}[\varphi(\bm y')]}.\\
        &\geq\E_{\bm y\sim \mathcal P_{k'}}[\tilde\ell(\bm y)]-\E_{\bm y\sim \mathcal P_k}[\tilde\ell(\bm y')]\\
        &=\inner{\E_{\bm y\sim \mathcal P_{k'}}[\nabla_{\bm x} \ell(\bm x_{k'}^\ast;\bm y)]-\E_{\bm y\sim \mathcal P_k}[\nabla_{\bm x} \ell(\bm x_{k'}^\ast;\bm y)]}{\bm x_k^\ast-\bm x_{k'}^\ast}.\label{eq:drift2}
    \end{align}

    By combining \cref{eq:drift1,eq:drift2}, we obtain
    \begin{align}
        \mu\|\bm x_k^\ast-\bm x_{k'}^\ast\|^2
        &\leq \inner{\E_{\bm y\sim \mathcal P_k}[\nabla_{\bm x} \ell(\bm x_k^\ast;\bm y)]-\E_{\bm y\sim \mathcal P_k}[\nabla_{\bm x} \ell(\bm x_{k'}^\ast;\bm y)]}{\bm x_k^\ast-\bm x_{k'}^\ast}\\
        &\leq L_{1,1}\|\bm x_k^\ast-\bm x_{k'}^\ast\|W_1(\mathcal P_k,\mathcal P_{k'})+\inner{\E_{\bm y\sim \mathcal P_k}[\nabla_{\bm x} \ell(\bm x_k^\ast;\bm y)]-\E_{\bm y\sim \mathcal P_{k'}}[\nabla_{\bm x} \ell(\bm x_{k'}^\ast;\bm y)]}{\bm x_k^\ast-\bm x_{k'}^\ast}\\
        &\leq L_{1,1}\|\bm x_k^\ast-\bm x_{k'}^\ast\|W_1(\mathcal P_k,\mathcal P_{k'}),
    \end{align}
    where \cref{lem:gradient_expectation,lem:optimality_property} are used in the last inequality.
\end{proof}



\subsection{Omitted Proofs in Tracking Error Analysis}

\begin{lemma}\label{lem:recursion_lem}
    Let $a_0\in\R$, $a_{i}\geq0$, and $a\coloneqq\sum_{i=1}^na_{i}<1$.
    Suppose $\{e_k\}_{k=0}^\infty$ satisfies 
    \begin{align}
        e_k\leq\sum_{i=1}^na_{i}e_{k-i}+a_0\label{eq:recursion_lem1}
    \end{align}
    for all $k\geq n$.
    Then, there exists $C\geq0$ such that, for all $k\geq0$,
    \begin{align}
        e_k\leq \frac{a_0}{1-a}+Ca^{k/n}.\label{eq:recursion_lem2}
    \end{align}
\end{lemma}

\begin{proof}
    Let 
    \begin{align}
        \tilde e_k\coloneqq\begin{dcases*}
            e_k-\frac{a_0}{1-a}&if $k\leq n-1$\\
            \sum_{i=1}^na_{i}\tilde e_{k-i}&if $k\geq n$.
        \end{dcases*}\label{eq:recursion_lem3}
    \end{align}
    We first prove 
    \begin{align}
        e_k\leq\tilde e_k+\frac{a_0}{1-a}\label{eq:recursion_lem4}
    \end{align}
    for all $k\geq0$ by induction.
    For $k\leq n-1$, \cref{eq:recursion_lem4} holds trivially.
    Assume that there exists $k'\geq n-1$ such that \cref{eq:recursion_lem4} holds for all $k\leq k'$.
    Then, for $k=k'+1$, we have
    \begin{align}
        e_{k'+1}
        &\leq\sum_{i=1}^na_{i}e_{k'+1-i}+a_0\\
        &\leq\sum_{i=1}^na_{i}\prn*{\tilde e_{k'+1-i}+\frac{a_0}{1-a}}+a_0\\
        &=\tilde e_{k'+1}+\frac{a_0}{1-a}\sum_{i=1}^na_{i}+ a_0\\
        &=\tilde e_{k'+1}+\frac{a_0}{1-a},
    \end{align}
    where the first inequality follows from \cref{eq:recursion_lem1} and the second from the induction hypothesis.

    Next, we prove that, by taking $C\coloneqq\max\{0,\tilde e_0,\tilde e_1,\ldots,\tilde e_{n-1}\}\geq0$, we have
    \begin{align}
        \tilde e_k\leq Ca^{k/n}\label{eq:recursion_lem5}
    \end{align}
    for all $k\geq 0$ by induction.
    For $k\leq n-1$, \cref{eq:recursion_lem5} holds trivially.
    Assume that there exists $k'\geq n-1$ such that \cref{eq:recursion_lem5} holds for all $k\leq k'$.
    Then, for $k=k'+1$, we have
    \begin{align}
        \tilde e_{k'+1}
        =\sum_{i=1}^na_{i}\tilde e_{k'+1-i}
        &\leq C\sum_{i=1}^na_{i}a^{(k'+1-i)/n}\\
        &\leq C\sum_{i=1}^na_{i}a^{(k'+1-n)/n}\\
        &= Ca^{(k'+1)/n},
    \end{align}
    where the first inequality follows from the induction hypothesis, the second from $a<1$.

    Combining \cref{eq:recursion_lem4,eq:recursion_lem5} completes the proof.
\end{proof}



\section{Omitted Proofs in \cref{sec:choices_of_coefficients}}

\subsection{Proof of \cref{lem:alpha_norm}}\label{sec:proof_of_lem_alpha_norm}
\begin{proof}
    The matrix $M$ in \cref{eq:alpha_matrix} has the following form:
    \begin{align}
        M
        &=\mqty(
            n&\frac{n^2}2+O(n)&\cdots&\frac{n^p}p+O(n^{p-1})\\
            \frac{n^2}2+O(n)&\frac{n^3}3+O(n^2)&\cdots&\frac{n^{p+1}}{p+1}+O(n^{p})\\
            \vdots&\vdots&&\vdots\\
            \frac{n^p}p+O(n^{p-1})&\frac{n^{p+1}}{p+1}+O(n^{p})&\cdots&\frac{n^{2p-1}}{2p-1}+O(n^{2p-2})
        ).
    \end{align}
    Let 
    \begin{align}
        \tilde M
        &\coloneqq\mqty(
            1&\frac12&\cdots&\frac1p\\
            \frac12&\frac13&\cdots&\frac1{p+1}\\
            \vdots&\vdots&&\vdots\\
            \frac1p&\frac1{p+1}&\cdots&\frac{1}{2p-1}
        ).
    \end{align}
    This is invertible and the inverse of $\tilde M$ has the following explicit form \citep[Lemma 2.1]{trench1966inversion}:
    \begin{align}
        ({\tilde M}^{-1})_{i,j}
        &=(-1)^{i+j}\frac{(p+i-1)!(p+j-1)!}{(i+j-1)(p-i)!(p-j)!((i-1)!)^2((j-1)!)^2},
    \end{align}
    which is non-zero for all $i,j=1,\ldots,p$.
    From this fact and the cofactor expansion, we can see that the inverse of $M$ has the following form:
    \begin{align}
        (M^{-1})_{i,j}
        &=({\tilde M}^{-1})_{i,j}n^{-i-j+1}+O(n^{-i-j}).
    \end{align}
    Replacing $M^{-1}$ in \cref{eq:alpha_matrix} by this expression, we have
    \begin{align}
        \alpha_{i}
        &=\sum_{j=1}^{p}i^{j-1}(M^{-1})_{j,1}
        =\sum_{j=1}^{p}i^{j-1}\prn*{({\tilde M}^{-1})_{j,1}n^{-j}+O(n^{-j-1})}.
    \end{align}
    The fundamental theorem of algebra guarantees that there exists $i\in[n]$ such that $\sum_{j=1}^{p}i^{j-1}({\tilde M}^{-1})_{j,1}\neq0$.
    Therefore, we have
    \begin{align}
        n\|\bm \alpha\|^2
        &=n\sum_{i=1}^n\prn*{\sum_{j=1}^{p}i^{j-1}\prn*{({\tilde M}^{-1})_{j,1}n^{-j}+O(n^{-j-1})}}^2\\
        &=\sum_{i=1}^n\sum_{j=1}^{p}\sum_{l=1}^{p}i^{j+l-2}\prn*{({\tilde M}^{-1})_{j,1}({\tilde M}^{-1})_{l,1}n^{-j-l+1}+O(n^{-j-l})}\\
        &=\sum_{j=1}^{p}\sum_{l=1}^{p}\prn*{\frac{n^{j+l-1}}{j+l-1}+O(n^{j+l-2})}\prn*{({\tilde M}^{-1})_{j,1}({\tilde M}^{-1})_{l,1}n^{-j-l+1}+O(n^{-j-l})}\\
        &=\sum_{j=1}^{p}\sum_{l=1}^{p}\frac{({\tilde M}^{-1})_{j,1}({\tilde M}^{-1})_{l,1}}{j+l-1}+O(n^{-1}),
    \end{align}
    which implies $\|\bm \alpha\|=O(1/\sqrt n)$.
    Furthermore, $\|\bm \alpha\|_1=O(1)$ follows from the Cauchy--Schwarz inequality.
\end{proof}

\subsection{Proof of \cref{lem:smooth_solution}}\label{sec:proof_of_lem_smooth_solution}
\begin{proof}
    Define the forward shift operator $S$ acting on a sequence $\{\bm a_k\}_{k=0}^\infty$ by $(S\{\bm a_k\})_k\coloneqq\bm a_{k+1}$.
    For the solution sequence $\{\bm x_k^\ast\}_{k=0}^\infty$, consider the operator polynomial
    \begin{align}
        P(S)\coloneqq\sum_{i=1}^n\alpha_{i}S^{n-i}-S^n.
    \end{align}
    so that $X_k^\ast\bm\alpha-\bm x_k^\ast=\prn*{P(S)\{\bm x_k^\ast\}}_{k-n}$.

    First, we show that the polynomial $P(z)$ has a root at $z=1$ with multiplicity at least $p$.
    This is equivalent to showing that $P^{(j)}(1)=0$ for all $j\in[0,p-1]$.
    Differentiating $P(z)$ gives
    \begin{align}
        P^{(j)}(1)
        =\dv[j]{z}\prn*{\sum_{i=1}^n\alpha_{i} z^{n-i}-z^n}\Bigg|_{z=1}
        &=\sum_{i=1}^{n}\alpha_iQ_j(-i)-Q_j(0)
    \end{align}
    for all $j\in[0,p-1]$, where $Q_j(z)\coloneqq\prod_{k=0}^{j-1}(n-k+z)$.
    Rewriting this by defining $Q_j(z)\eqqcolon\sum_{m=0}^{p-1} q_{j,m} z^m$ gives
    \begin{align}
        P^{(j)}(1)
        &=\mqty(q_{j,0}&q_{j,1}&\cdots&q_{j,p-1})\mqty(
            (-1)^0&(-2)^0&\cdots&(-n)^0\\
            (-1)^1&(-2)^1&\cdots&(-n)^1\\
            \vdots&\vdots&&\vdots\\
            (-1)^{p-1}&(-2)^{p-1}&\cdots&(-n)^{p-1}
        )\mqty(\alpha_1\\\alpha_2\\\vdots\\\alpha_n)-q_{j,0}\\
        &=\mqty(q_{j,0}&q_{j,1}&\cdots&q_{j,p-1})\Phi\bm\alpha-q_{j,0}.
    \end{align}
    This is zero because \cref{eq:sum_of_alpha} implies $\Phi\bm\alpha=\bm\phi(0)=\mqty(1&0&\cdots&0)^\top$.

    Next, we give a bound on $\|\prn*{P(S)\{\bm x_k^\ast\}}_{k-n}\|$.
    Since $P(z)$ has a root at $z=1$ with multiplicity at least $q\in[1,p]$, there exist $b_0,\ldots,b_{n-q}\in\R$ such that 
    \begin{align}
        P(z)
        &=(1-z)^q\sum_{i=0}^{n-q}b_iz^i.\label{eq:smooth_solution1}
    \end{align}
    Applying $P(S)$ to the sequence $\{\bm x_k^\ast\}_{k=0}^\infty$ and using \cref{eq:smooth_solution1} gives
    \begin{align}
        \prn*{P(S)\{\bm x_k^\ast\}}_{k-n}
        &=\prn*{(1-S)^q\sum_{i=0}^{n-q}b_iS^i\{\bm x_k^\ast\}}_{k-n}
        =\prn*{\sum_{i=0}^{n-q}\sum_{j=0}^qb_i(-1)^j\binom{q}{j}S^{i+j}\{\bm x_k^\ast\}}_{k-n}\\
        &=\sum_{i=0}^{n-q}\sum_{j=0}^qb_i(-1)^j\binom{q}{j}\bm x_{k-n+i+j}^\ast.
    \end{align}
    Taking the norm and using the triangle inequality yields
    \begin{align}
        \norm{\prn*{P(S)\{\bm x_k^\ast\}}_{k-n}}
        &\leq\sum_{i=0}^{n-q}|b_i|\norm{\sum_{j=0}^q(-1)^j\binom{q}{j}\bm x_{k-n+i+j}^\ast}.
    \end{align}
    By a property of the finite-difference (see \citet[Section 2.1]{kamijima2025simple}), we have
    \begin{align}
        \norm{\sum_{j=0}^q(-1)^j\binom{q}{j}\bm x_{k-n+i+j}^\ast}
        &\leq h^q\sup_{t\geq0}\|(\bm x^\ast)^{(q)}(t)\|.
    \end{align}
    Therefore, 
    \begin{align}
        \norm{\prn*{P(S)\{\bm x_k^\ast\}}_{k-n}}
        &\leq\sum_{i=0}^{n-q}|b_i|h^q\sup_{t\geq0}\|(\bm x^\ast)^{(q)}(t)\|,\label{eq:smooth_solution2}
    \end{align}

    Finally, we bound $\sum_{i=0}^{n-q}|b_i|$.
    Substituting $P(z)$ in \cref{eq:smooth_solution1} by the definition of $P$ gives
    \begin{align}
        \sum_{i=0}^{n-q}b_iz^i
        &=\prn*{\sum_{i=1}^n\alpha_{i} z^{n-i}-z^n}(1-z)^{-q}\\
        &=\prn*{\sum_{i=1}^n\alpha_{i} z^{n-i}-z^n}\prn*{\sum_{l=0}^\infty\binom{q+l-1}{q-1}z^l},
    \end{align}
    where we used the Taylor expansion of $(1-z)^{-q}$ in the last equality.
    By comparing the coefficients of $z^i$ on both sides for $i\in[0,n-q]$, we have
    \begin{align}
        b_i
        =\sum_{l=0}^i\binom {q+l-1}{q-1}\alpha_{n-i+l}.
    \end{align}
    Thus, the triangle inequality gives
    \begin{align}
        \sum_{i=0}^{n-q}|b_i|
        &\leq\sum_{i=0}^{n-q}\sum_{l=0}^i\binom {q+l-1}{q-1}|\alpha_{n-i+l}|
        =\sum_{l=0}^{n-q}\binom {q+l-1}{q-1}\sum_{i=l}^{n-q}|\alpha_{n-i+l}|.
    \end{align}
    Using $\sum_{i=l}^{n-q}|\alpha_{n-i+l}|\leq\|\bm\alpha\|_1$, we obtain
    \begin{align}
        \sum_{i=0}^{n-q}|b_i|
        &\leq\sum_{l=0}^{n-q}\binom {q+l-1}{q-1}\|\bm \alpha\|_1
        =\binom {n}{q}\|\bm \alpha\|_1.\label{eq:smooth_solution3}
    \end{align}
    Inserting \cref{eq:smooth_solution3} into \cref{eq:smooth_solution2} completes the proof.
\end{proof}